\documentclass[11pt]{article}

\usepackage{amsmath,amsfonts,amsthm,amscd,amssymb,graphicx}

\usepackage{subfigure}
\usepackage{xcolor}

\numberwithin{equation}{section}
\usepackage{hyperref}

\newtheorem{theorem}{Theorem}[section]

\newtheorem{lemma}[theorem]{Lemma}

\newtheorem{proposition}[theorem]{Proposition}

\newcommand{\RR}{\mathbb{R}}

\newcommand{\cE}{\mathcal{E}}

\newcommand{\ZZ}{{\mathbb Z}}
\newcommand{\TT}{{\mathbb T}}

\newcommand{\e}{{\varepsilon}}

\def\beq{\begin{equation}}
\def\eeq{\end{equation}}
\def\bb1{{1\!\!1}}
\def\Ff{\widehat{f}}
\def\Fg{\widehat{g}}
\def\FcE{\widehat{\cE}}
\def\Frho{\widehat{\rho}}
\def\FE{\widehat{E}}
\def\FS{\widehat{S}}
\def\FG{\widehat{G}}

\begin{document}

\title{Remarks on Landau damping} 

\author{Toan T. Nguyen\footnotemark[1]
}

\maketitle

\footnotetext[1]{Penn State University, Department of Mathematics, State College, PA 16802. Email: nguyen@math.psu.edu. }

\vspace{-.3in}

\begin{center}
	\emph{In honor of Dang Duc Trong
		on the occasion of his sixtieth birthday 
%		\\a kind mentor to generations of Vietnamese mathematicians
		%\\as a token of gratitude and friendship 
	}
\end{center}
%\bigskip

\begin{abstract}
	
We provide few remarks on nonlinear Landau damping that concerns decay of the electric field in the classical Vlasov-Poisson system near spatially homogenous equilibria. In particular, this includes the analyticity framework, \`a la Grenier-Nguyen-Rodnianski, for non specialists, treating the analytic case studied by Mouhot-Villani, among other remarks for plasmas confined on a torus and in the whole space. Finally, we also establish the nonlinear Landau damping for a family of Vlasov-Riesz systems, which are new and surprisingly include the borderline Vlasov-Dirac-Benney system in the sharp analytic spaces.
 
\end{abstract}

%%%%%%%%%%%%%%%%%%%%%%

\section{Introduction}

%%%%%%%%%%%%%%%%%%%%%

Of great interest in plasma physics is to establish the large time behavior of solutions to the following classical Vlasov-Poisson system 
\begin{equation}
\label{VP1}
\partial_t f + v\cdot \nabla_x f + E \cdot \nabla_v f = 0
\end{equation} 
\begin{equation}\label{VP2}
E = -\nabla_x \phi, \qquad -\Delta_x \phi = \rho  
\end{equation}
modeling the dynamics of excited electrons confined on a torus $\TT_x^3\times \RR^3$ or in the whole space $\RR^3_x\times \RR^3_v$, in which $\rho(t,x) = \int_{\RR^3} f(t,x,v)\; dv-n_{\mathrm{ion}}$ denotes the charged density, and 
$ n_{\mathrm{ion}}$ is a non-negative constant representing the uniform ions background. The Cauchy problem is rather classical, going back to the works \cite{LionsPerthame,Pfaffelmoser,Schaeffer}, which assert that smooth initial data $f(0,x,v)$ with finite moments give rise to global-in-time smooth solutions. However, their large time behavior is largely open due to the presence of plasma echoes and rich underlying physics, which we shall now discuss.

In this paper, we shall develop an analyticity framework to establish Landau damping near Penrose stable equilibria. As an application, the results apply to a family of Vlasov systems with Riesz-type interaction potentials, namely the classical Vlasov equation \eqref{VP1} coupled with 
\begin{equation}\label{VP2Riesz}
E = -\nabla_x (-\Delta_x)^{-\alpha/2} \rho  
\end{equation}
 for all $\alpha \in [0,2]$. The classical Vlasov-Poisson system \eqref{VP1}-\eqref{VP2} corresponds to the case when $\alpha=2$, while the Vlasov-Dirac-Benney system is the case when $\alpha=0$. The case when $\alpha=0$ appears to be the borderline case for analytic regularity, as it is ill-posed in any Gevrey-$\frac{1}{\gamma}$ classes with $\gamma>1$, see \cite{BardosBesse, DanielToan, HanKwanR}. 

\subsection{Phase mixing}\label{secPM}

Phase mixing is a damping mechanism due to shearing in the phase space $\TT^3_x\times\RR^3_v$ by the free-transport dynamics (see Figure \ref{fig-PM}), 
\begin{equation}\label{free}
\partial_t f^0 + v \cdot \nabla_x f^0 = 0
\end{equation}
leading to decay of macroscopic quantities such as the charged density 
$\rho^0(t,x) = \int_{\RR^3} f^0(t,x,v) \; dv.
$ Indeed, the transport dynamics or shearing of each elementary mode $e^{ik\cdot x}$, for $k\in \ZZ^3\setminus\{0\}$, creates fast oscillation $e^{-ikt\cdot v}$ in $v$, which leads to rapid decay of macroscopic quantities, thanks to velocity averaging. The decay costs derivatives, and it is exponentially fast  in $|kt|$ if data are analytic and polynomially fast if data only have finite regularity. Note that the zero mode $k=0$ does not decay, but is conserved in time. For sake of simplicity, we always assume that the total charged density $\int \rho(t,x)\; dx =0$ (i.e. plasmas are globally neutralized).  

The phase mixing remains valid in the whole space $\RR^3_x\times \RR^3_v$ for each elementary wave packet of spatial frequency $k\in \RR^3\setminus\{0\}$. In addition, as the particles are unconfined and transported along particle trajectories $(x,v)\mapsto (x+vt,v)$ in the whole space, the contracting in volume $dv = t^{-3} dx_t$ yields an additional decay of macroscopic quantities at order of $t^{-3}$, often referred to as dispersion by the transport dynamics. The latter transport dispersion requires initial data to be sufficiently localized in space.

Consequently, the free electric field (i.e. those that are generated by the free transport dynamics) decays, through the Poisson equation $E^0 = \nabla_x \Delta_x^{-1}\rho^0$, as in \eqref{VP2}, exponentially fast for analytic data and polynomial fast for Sobolev data. In the whole space, the electric field decays at rate of order $t^{-2}$, which can be seen via the standard elliptic estimate $\| E^0\|_{L^\infty}\lesssim \|\rho^0\|^{1/3}_{L^1}\|\rho^0\|_{L^\infty}^{2/3}$, upon using the transport dispersion and mass conservation for $\rho^0(t,x)$. 

\begin{figure}\label{fig-PM}
\centering	
\subfigure{
		\includegraphics[scale=.17]{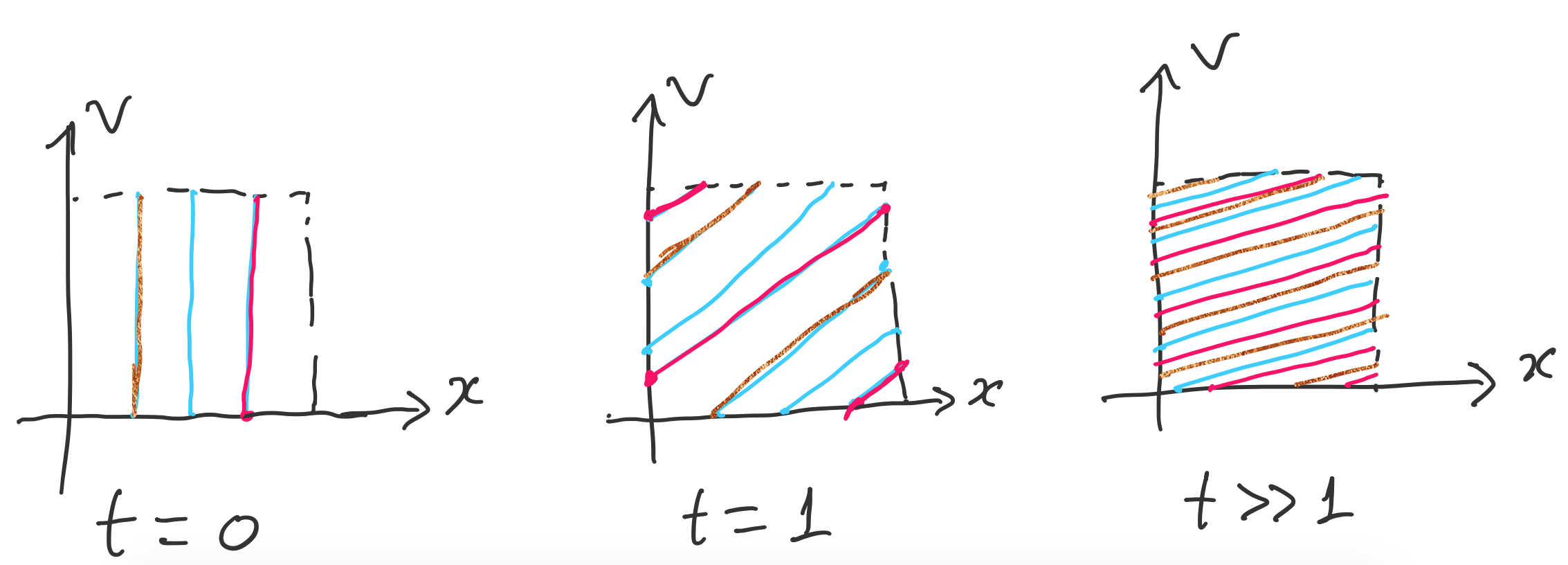}}
	%\caption{$\partial_t f + v \cdot \nabla_x f =0$}
	%\quad %$\Longrightarrow$ %
	\hspace{.4in}
	%\quad
	\subfigure{
		\includegraphics[scale=.24]{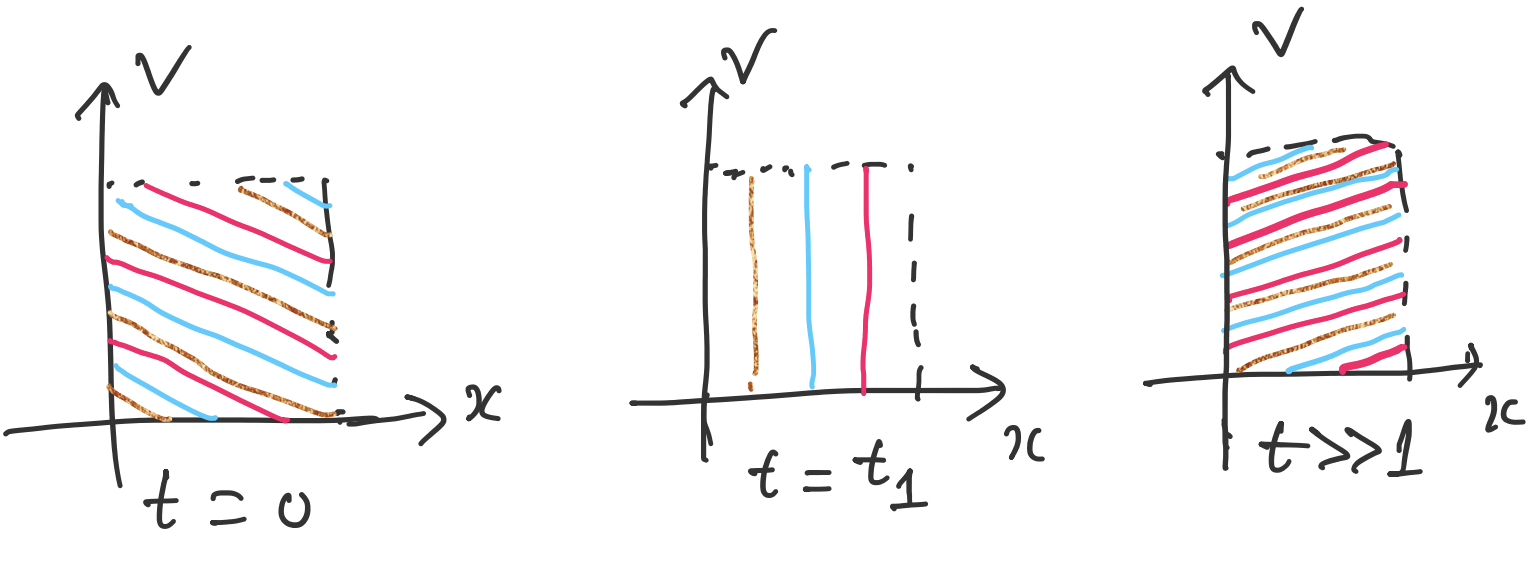}}
	\caption{\em  Depicted are particle trajectories of free gas in the phase space. As time $t$ increases, particles that may be unmixed at $t=0$ (left) or at $t=t_1$ (right) will be mixed in the large time. Averaging in $v$ then yields rapid decay due to oscillation.}
\end{figure}

\subsection{Plasma echoes}\label{sec-pecho}

Returning to the nonlinear Vlasov-Poisson system \eqref{VP1}-\eqref{VP2}, we may expect that the phase mixing mechanism remains dominant and the electric field $E$ is damped by phase mixing as described above. However, the electric field may become much larger due to plasma echoes. Indeed, inductively, the quadratic nonlinear term $-E\cdot \nabla_v f$ generates an iterative density 
\begin{equation}\label{denR}
\rho^R = -\int_0^t \int_{\RR^3} [E\cdot \nabla_v f](s,x-v(t-s),v)\; dvds 
\end{equation}
which is obtained by viewing the Vlasov equation \eqref{VP1} as the free transport with forcing $-E\cdot \nabla_v f$. Here and in what follows, $[Ef](t,x,v) = E(t,x)f(t,x,v)$. Therefore, using the Poisson equation \eqref{VP2} and noting $\nabla_v - t\nabla_x$ is invariant under \eqref{free}, the nonlinear electric field of \eqref{VP1}-\eqref{VP2} is of the form
\begin{equation}
\label{Eform}
E = E^0(t,x) + E^R(t,x)
\end{equation}  
where $E^0(t,x)$ is the free electric field generated by \eqref{free}, and $E^R=  -\nabla_x\Delta_x^{-1}\nabla_x \cdot \cE^R$, having set  
\begin{equation}
\label{ER}
\cE^R := \int_0^t \int_{\RR^3} (t-s)[Ef](s,x-v(t-s),v)\; dvds.
\end{equation}
Effectively, up to a Fourier multiplier of order zero $\nabla_x\Delta_x^{-1}\nabla_x$, the electric field is inductively computed by the collective integral of $[Ef]$ along the free transport, which may be large due to resonant interaction between different energy levels (i.e. different modes) that may cancel out oscillation in $v$, leading to plasma echoes. Indeed, for $k\in \ZZ^3, \eta \in \RR^3$, letting 
$\FE_k(t)$ and $\Ff_{k,\eta}$ denote the Fourier transform of $E(t,x)$ in $x$ and of $f(t,x,v)$ in $(x,v)$, respectively, we obtain 
\begin{equation}
\label{ER-Fr}
\FcE^R_k(t) = \sum_{k_1+k_2 = k}\int_0^t (t-s) \FE_{k_1}(s) \Ff_{k_2, k(t-s)}(s)ds,
\end{equation}
in which $k_1,k_2\in \ZZ^3$. 
To leading order, we may expect that $\Ff_{k,\eta}(t)$ remains sufficiently localized near frequency $(k,\eta + kt)$ dictated by \eqref{free}, and therefore $\Ff_{k_2, k(t-s)}(s)$ is localized near $(k_2,kt - k_1s)$, recalling $k_1+k_2=k$. In particular, locally near $kt = k_1 s$, we may have an amplification from an electric pulse at $(s,k_1)$ to that at $(t,k)$ by a factor of order 
$(t-s)\approx \frac{(k_1-k)\cdot k_1}{|k_1|^2} t$, which could be arbitrarily large in the large time. Note that since $s<t$, the energy transfer between electric pulses is from high to low (or more precisely, from $(s,k_1)$ to $(t,k)$ with $|k_1|^2>k\cdot k_1$). 

The physics of plasma echoes goes as follows. Start with an elementary wave $e^{ik_1\cdot x} e^{i\eta_1\cdot v} g_1(v)$ for some analytic function $g_1(v)$ (i.e. exponentially localized in frequency). This leads to an electric pulse $E_1(t)$ which is exponentially localized at the critical time $t_1 = \eta_1/k_1$ (by phase mixing, noting oscillation in $v$ vanishes at $\eta_1 = k_1 t$). Inject another elementary wave $e^{ik_2\cdot x} e^{i\eta_2\cdot v} g_2(v)$ whose electric field $E_2(t)$ is exponentially localized at a different critical time $t_2 = \eta_2/k_2$. The quadratic interaction between the two modes gives rise to a third wave, computed by \eqref{ER-Fr}, whose is again exponentially localized at a new critical time $t_3 =(\eta_1+\eta_2)/(k_1+k_2)$, which could be arbitrarily long after the first two waves have damped away (e.g., taking $k_2=-k_1+1, \eta_1 = k_1$, and $\eta_2 = 2k_1$ with $k_1\gg1$). This phenomenon is often referred to as a plasma echo, and the process leads to an infinite cascade of plasma echoes. See \cite{Gould} for more discussion on the physics, and also \cite{Bed2, GNR2} for some mathematical justification of plasma echoes. 

\subsection{Suppression of echoes}\label{sec-echo}

In view of the electric field formulation \eqref{Eform}-\eqref{ER}, there are no loss of derivatives, but apparent lack of decay due to the presence of $(t-s)$ in the time integration \eqref{ER}. However, in the case of plasmas on a torus, decay costs derivatives, and a convenient way to deal with the linear growth in time is to work with analytic functions which are exponentially localized in time due to phase mixing mechanism, see Section \ref{secPM}. This was first established in the celebrated work by Mouhot-Villani \cite{MV} for analytic data near spatially homogenous equilibria, and subsequently extended to include Gevrey data \cite{BMM-apde,GNR1,IPWW3}, namely those that are exponentially localized in $\langle k, \eta\rangle^\gamma$, for $\gamma\in [1/3,1]$ (noting the analytic case corresponds to $\gamma=1$). The fact that the Gevrey-$\frac13$ regularity was needed may be seen formally from \eqref{ER-Fr}, which leads to an amplification of order $(t-s)/|k|\approx t/|k|^2 \approx |\eta|/|k|^3$ after each iteration from $k+1$ to $k$ (noting the extra factor of $1/|k|$ is due to the integration in $s$ locally near resonant times $kt \sim k_1 s$, with $k_1=k+1$). This amplification accumulates and leads to a growth in Gevrey-$\frac13$ norms as observed in \cite{Bed2}. 

In the whole space, the transport dispersion yields extra decay without costing derivatives (but spatial localization), and one may hope that plasma echoes could be suppressed. Indeed, one may attempt to bound the collective integral \eqref{ER} by 
\begin{equation}
\label{ER-est}
\begin{aligned}
\cE^R(t,x) \lesssim \int_0^t (t-s)^{-2}\int_{\RR^3} |E(s,y)f(s,y,\frac{x-y}{t-s})|\; dyds,
\end{aligned}
\end{equation}
upon introducing $y = x - v(t-s)$. This leads to a dispersive decay of order $t^{-2}$ for the electric field, provided that $f(t,x,v)$ remains sufficiently localized in $x$. This was indeed established in the classical work by Bardos-Degond \cite{BD} for the case near vacuum (e.g., $f$ is sufficiently localized and small in $C^1$ norms). In the case near spatially homogenous equilibria, whether plasma echoes could be suppressed remains an open problem, since $f$ is no longer localized in space due to spatial homogenity of background equilibria. In fact, the non-localization or long-range interaction leads to plasma oscillations, a new regime in which the electric field oscillates and disperses at a much slower rate of order $t^{-3/2}$. Phase mixing is no longer available in this regime, which causes several new issues; see Section \ref{sec-linLD} for more discussion.

Finally, we mention that plasma echoes can also be suppressed by screening out the long-range interaction in the whole space (e.g., replacing Coulomb's by Yukawa-type potential). This excludes the mentioned plasma oscillations, and Landau damping near spatially homogenous equilibria can be established for data with finite regularity \cite{BMM-cpam, HKNR2}. The suppression of echoes can also be done via enhanced dissipation \cite{Bed1,CLN}.

\subsection{Linear Landau damping}\label{sec-linLD}

Landau damping concerns decay of the electric field in a non-trivial non-equilibrium state \cite{Landau}. Namely, consider a spatially homogenous equilibrium $\mu(v)$ at neutrality $\int \mu(v)\; dv = n_{\mathrm{ions}}$ with vanishing electric field. Returning to the electric field formulation \eqref{Eform}-\eqref{ER}, with $f$ replaced by $\mu + f$, we note that the collective electric field $\cE^R$ now contributes a linear term 
\begin{equation}
\label{ER-lin}
\cE[E] = \int_0^t \int_{\RR^3} (t-s)[E\mu](s,x-v(t-s),v)\; dvds.
\end{equation}
This linear term may be put into the left of \eqref{Eform}, leading to 
\begin{equation}
\label{Eform1}
E + \nabla_x\Delta_x^{-1}\nabla_x \cdot \cE[E]  = E^0(t,x) + E^R(t,x)
\end{equation}  
where $E^0, E^R$ are defined as in \eqref{Eform}-\eqref{ER}, which corresponds to the free electric field and nonlinear integral term \eqref{ER}. We stress that this is a formulation of the electric field for the nonlinear system \eqref{VP1}-\eqref{VP2} near equilibria $\mu(v)$. Let us first recall the following result concerning the linearized operator, precisely $1 + \nabla_x\Delta_x^{-1}\nabla_x \cdot \cE[\cdot]$, appearing on the left hand side of \eqref{Eform1}.

\begin{theorem}[Linear Landau damping]\label{theo-linLD} Let $\mu(v)$ be any radial, sufficiently smooth, and rapidly decaying equilibrium, and let $G(t,x)$ be the spacetime Green function of the linearized operator in \eqref{Eform1}.
Then, for each $k\in \RR^3$, the Fourier transform of $G(t,x)$ satisfies 

\begin{equation}\label{decomp-FG} 
\FG_k(t) = \delta(t) + \sum_\pm \FG^{osc}_{k,\pm}(t)  +   \FG^{r}_k(t) 
\end{equation}
where $
\FG^{osc}_{k,\pm}(t) = e^{\lambda_\pm(k)t} a_\pm(k)
$
for some sufficiently smooth Fourier multipliers $\lambda_\pm(k), a_\pm(k)$ that are compactly supported near the origin, and of Klein-Gordon's type $\Im \lambda_\pm(k) \approx \pm \sqrt{1+|k|^2}$, with $\Re \lambda_\pm(k)\le 0$. In addition, there holds   
\begin{equation}\label{bd-Gr}
|\FG^{r}_k(t)| \lesssim |k|^3 \langle k\rangle^{-4} \left\{\begin{aligned}
\langle kt\rangle^{-N_0} , &\qquad \mbox{if $\mu(v)$ has limited regularity of order $N_0$}
\\e^{-\theta_0\langle kt\rangle}, &\qquad \mbox{if $\mu(v)$ is analytic with radius of order $\theta_0$} 
\end{aligned}\right.
\end{equation}
uniformly in $k\in \RR^3$. 
\end{theorem}

Let us comment on the theorem, which has been proven, up to some variants, by several authors, e.g., \cite{MV, GNR1,HKNR2,HKNR3, HKNR4, BMM-lin,IPWW2,Toan}. The same results are also obtained in low dimensions under the so-called Penrose stability condition to exclude possible unstable modes (which are absent in dimensions three and higher for radial equilibria, \cite{MV}). We note that the Green function $G(t,x)$ is constructed in such a way that the nonlinear electric field can now be solved, iteratively from \eqref{Eform1}, yielding   
\begin{equation}
\label{nonE}
E = G\star_{t,x}E^0(t,x) + G\star_{t,x}E^R(t,x)
\end{equation}  	
in which we extend by zero the functions defined in negative $t$ (cf. \eqref{Eform}). 
The boundedness of $G$ as an operator from $L^2_{t,x}$  to $L^2_{t,x}$ was proven in \cite{MV}, while the pointwise formulation in term of the Green function was first established in \cite{GNR1} for the torus case and \cite{HKNR2} for the whole space case, which appears convenient in nonlinear analyses. The fact that the Green function consists of a Klein-Gordon's oscillatory component is obtained independently in \cite{HKNR3} for general equilibria and in \cite{BMM-lin} for Gaussians. In fact, there is a non-trivial survival threshold of wavenumbers $\kappa_0$, below which the oscillatory behavior persists for compactly supported or relativistic equilibria \cite{Toan,HKNR4}. Above the threshold, the oscillatory component is damped, i.e., $\Re \lambda_\pm(k)<0$, in a manner that is very sensitive to the vanishing or decaying rate of equilibria at their maximal velocity: the faster it vanishes, the weaker this damping rate is. For radial and positive equilibria\footnote{this formula is stated in dimension three. In dimension one, it reads $\Re \lambda_\pm(k) \approx |k|^{-2}\partial_v \mu(\frac{1}{|k|})$, which is often used to interpret the physical meaning of Landau damping versus Landau growth, depending on more or less particles in front of the wave. Note that $\partial_v\mu(v)$ is evaluated at $v = \frac{1}{|k|}$, namely at the resonant interaction between particles and waves.}, $\Re \lambda_\pm(k) \approx -|k|^{-3}\mu(\frac{1}{|k|})$ for $|k|\ll1$, which is often referred to as the original Landau damping, as documented in the physical literature \cite{Landau, Trivelpiece}. For general radial equilibria, a similar damping rate can also be derived, see \cite{Toan}. We note that while the oscillatory behavior was absent in the case of torus for positive equilibria \cite{MV}, they may be found for compactly supported or relativistic equilibria on a sufficiently large torus. Consequently, we note that linear Landau damping or decay of the linearized electric field $E = G\star_{t,x}E^0(t,x)$ now follows straightforwardly from that of the Green function and phase mixing estimates on $E^0$, see, e.g., \cite{HKNR3,Toan}.   

\subsection{Nonlinear Landau damping}

To establish nonlinear Landau damping, one must analyze plasma echoes as appeared in the formulation \eqref{nonE}, see also Section \ref{sec-echo}. Mouhot-Villani \cite{MV} and several subsequent works  \cite{BMM-apde, GNR1,GNR2,GIa, IPWW3} were able to suppress echoes by analytic and Gevrey regularity in the case of a torus, with which phase mixing remains dominant, leading to exponential decay of the electric field. 

In the case of the whole space, plasma echoes can again be suppressed via the extra decay from transport dispersion and nonlinear Landau damping holds for data with finite regularity \cite{BMM-cpam, HKNR2,TrinhLD} for screened Vlasov-Poisson systems. The screening turns out to be fundamental, which excludes plasma oscillations, namely the absence of oscillatory component $\FG^{osc}_{k,\pm}(t)$ in linear Landau damping \eqref{decomp-FG}. For otherwise, plasma oscillations are always present for all rapidly decaying equilibria in the whole space, recalling the real part of the dispersion relation $\Re \lambda_\pm(k) \approx - |k|^{-3}\mu(\frac{1}{|k|})$ is vanishing as $|k|\to 0$ (in fact, identically zero for all $k$ below the survival threshold, \cite{Toan}). Let us also mention a recent work \cite{IPWW} that treats a special equilibrium $\mu(v) = \langle v\rangle^{-4}$, for which Landau damping incidentally coincides phase mixing with exponential damping for each Fourier mode, which again excludes plasma oscillations. For general equilibria in the whole space, phase mixing is no longer available in the low frequency regime, and the Klein-Gordon's dispersion of plasma oscillations must be exploited to control plasma echoes. This was notably open even for the case of plasmas near vacuum, exactly due to lack of phase mixing, which is recently resolved in \cite{ToanVKG1}. The Landau damping for the Vlasov-Klein-Gordon systems near spatially homogenous equilbria is also established in \cite{ToanVKG2}.

%%%%%%%%%%%%%%%%%%%%

\section{Analyticity framework}\label{sec-analytic}

%%%%%%%%%%%%%%%%%%%%

In this section, we present the analyticity framework, \`a la Grenier-Nguyen-Rodnianski \cite{GNR1}, to establish nonlinear Landau damping, focusing only the analytic case. The framework was developed to treat both analytic and Gevrey data in an elementary way. However, if one were to focus only the analytic case as was in Mouhot-Villani \cite{MV}, the framework is even more direct without exploiting much structure of the transport equation, a remark which we thought worth pointing it out, if not already apparent from \cite{GNR1}. In any case, the framework is self-contained, apart from the linear theory provided in the previous section, which may be of use in other contexts.  

Precisely, we consider the nonlinear Vlasov-Poisson system \eqref{VP1}-\eqref{VP2} near spatially homogenous, real analytic, and Penrose stable equilibria $\mu(v)$ on the torus $\TT_x^d \times \RR^d_x$, $d\ge 1$ (recalling that all radial and positive equilibria are Penrose stable in dimensions $d\ge 3$, see \cite{MV}), namely 
\begin{equation}
\label{VP-pert}
\partial_t f + v \cdot \nabla_x f  = - E \cdot \nabla_v \mu  - E \cdot\nabla_v f, \qquad \nabla_x \cdot E = \rho, 
\end{equation} 
with initial perturbations  $f_0(x,v)$ of zero mass. Let $\Ff_{0,k,\eta}$ be the Fourier transform of $f_0(x,v)$ in $x,v$, respectively. Then, we have the following theorem. 

\begin{theorem}[Nonlinear Landau damping]\label{theo-LD}
	Let $\lambda_0 > 0$, and $\mu(v)$ be any spatially homogenous, real analytic, Penrose stable equilibria. Then, for any initial data $f_0$ of the system \eqref{VP-pert} with  
	$$ 
\epsilon_0:=	\sum_{|\alpha|\le d}\sum_{k\in \ZZ^d}  \int_{\RR^d}  e^{2\lambda_0 \langle k,\eta\rangle} 
	|\partial^\alpha_\eta \widehat f_{0,k,\eta}| d \eta 
	$$
	being sufficiently small, 
	the nonlinear Landau damping holds: $\|E(t)\|_{L^\infty}\lesssim \epsilon_0 e^{-\lambda_0 t}$ for all $t\ge 0$, and $f(t,x+vt,v)$ converges to a limit $f_\infty(x,v)$ in the large time.
\end{theorem}

To prove the theorem, we measure the analytic regularity of solutions using "generator functions". Namely,
for $z\ge 0$, we introduce
$$ 
Gen[g](z) := \sum_{|\alpha|\le d}\sum_{k\in \ZZ^d}  \int_{\RR^d} A_{k,\eta}(z)
|\partial^\alpha_\eta \widehat g_{k,\eta}|\; d \eta , \qquad A_{k,\eta}(z): = e^{z\langle k,\eta\rangle} \langle \eta\rangle^{\sigma} 
$$
where $\Fg_{k,\eta}$ denotes the Fourier transform of $g(x,v)$ in variables $x,v$, respectively, and $A_{k,\eta}(z)$ are classical analytic-Sobolev weight functions, with the usual bracket notation $\langle k,\eta\rangle = \sqrt{1+|k|^2+|\eta|^2}$. The additional Sobolev weight with $\sigma\ge 1$ is a classical way of quasi-linearizing the system to avoid double loss of derivatives (e.g. when treating $\partial_xE\partial_v f$). The only modification, compared to those introduced in \cite{GNR1}, is the use of $L^1$ norms, instead of $L^2$ norms where integration by parts was needed to avoid the apparent loss of derivatives in \eqref{VP-pert}. In the case of analytic data, one loss is allowed, and hence no structure of the transport equation is needed, for which we believe the framework can be adapted to other contexts. In addition, $L^1$ norms give a convenient algebra when treating products. Indeed, recalling $\widehat{(fg)}_{k,\eta} = \Ff_{k,\eta}\star_{k,\eta} \Fg_{k,\eta}$ and using the Hausdorff-Young inequality for convolution in $L^1$ norms, we have  
\beq \label{prodgen}
Gen[f g] \le Gen[f] Gen[g],
\eeq
upon noting the basic product property of the weight functions \begin{equation}\label{Aweight}
A_{k,\eta} \le A_{k',\eta'} A_{k-k', \eta-\eta'}
\end{equation}
for any $k,k'$ and $\eta, \eta'$. In addition, the norms are also well adapted to derivatives:
\beq \label{dxgen}
Gen[\partial_x g] \le \partial_z Gen[ g ], \qquad Gen[\partial_v g] \le \partial_z Gen[g] .
\eeq
Using the norm properties, we can turn the transport equation \eqref{VP-pert} into a transport inequality for the norms, which can then be solved for an appropriate time-dependent analyticity radius $z(t)$. Precisely, we first obtain the following. 

\begin{proposition}\label{prop-G0} Let $(E,f)$ solve \eqref{VP-pert}, and set $
	g(t,x,v) = f(t,x + v t, v).
	$ Then, there holds the following differential inequality
	$$
	\partial_t Gen[g(t)](z) \le C_0 F[E](t,z)  + (1+t) F[E](t,z) \partial_z Gen[g(t)](z)
	$$
	for some universal constant $C_0$, where 
	\begin{equation}\label{def-Nrho}
	F[E](t,z) :=  \sum_{k\in \ZZ^d\setminus\{0\}}  A_{k,kt}(z) |\FE_k(t)| .
	\end{equation}
	 
\end{proposition}

\begin{proof} By construction, we have 
	\begin{equation}\label{VP-g} 
	\partial_t g  = - E(t,x+vt) \cdot \nabla_v \mu(v) - E(t,x+vt) \cdot (\nabla_v - t \nabla_x) g .
	\end{equation}
The proposition thus follows, upon using the norm properties \eqref{prodgen}-\eqref{dxgen}, and the fact that the Fourier transform of $E(t,x+vt)$ in $x$ and $v$ is equal to $\FE_k(t)\delta_{\eta = kt}$. The ``norm'' $F[E](t,z)$ is therefore  the generator norm $Gen[\cdot]$ of the shifted electric field 
$E(t,x+vt)$, and the constant $C_0$ is the generator norm of $\nabla_v \mu$.	
\end{proof}

\subsection{Nonlinear iterative scheme}

The proof of Theorem \ref{theo-LD} now proceeds via an iterative scheme as follows. Fix sufficiently small constants $\theta_1, \delta>0$, and take a time-dependent analyticity radius $\lambda(t)=\lambda_0(1+(1+t)^{-\delta})$. We start with a bootstrap assumption on the electric field 
\begin{equation}\label{bootstrapE} 
F[E](t,\lambda(t)) \le \epsilon e^{-\theta_1 \langle t\rangle^{1-\delta}} 
\end{equation}
for all $t\in [0,T]$ for some positive $T$. We shall prove that, for sufficiently small $\epsilon$, the equality in \eqref{bootstrapE} is in fact never attained, and therefore the bootstrap argument can be continued for all times $t\ge 0$. The iterative scheme is divided into the following two steps. 

\begin{proposition}[Boundedness of generator norms]\label{prop-g}
Let $\epsilon_0 = Gen[f_0](2\lambda_0)$. Under the bootstrap assumption \eqref{bootstrapE}, there holds 
\begin{equation}
\label{bd-g}  Gen[g](t,\lambda(t)) \le \epsilon_0 + C_1 \epsilon,
\end{equation}
for some universal constant $C_1$. 
\end{proposition}
\begin{proof} Indeed, we use Proposition \ref{prop-G0} to evaluate the generator norm along $z = \lambda(t)$ (i.e. following the ``characteristic'' of the transport inequality), namely
	$$ \frac{d}{dt} Gen[g](t,\lambda(t)) \le C_0 F[E](t,\lambda(t))  + \Big[ \lambda'(t) + (1+t) F[E](t,\lambda(t)) \Big]\partial_z Gen[g(t)](\lambda(t)).$$
	Next, using the bootstrap assumption \eqref{bootstrapE} and the choice of the slowly decaying analyticity radius $\lambda(t)=\lambda_0(1+(1+t)^{-\delta})$, we have 
	\begin{equation}\label{cond-lambda}  
	\lambda'(t) + (1+t) F[E](t,\lambda(t)) \le 0. 
	\end{equation}
	which implies $\frac{d}{dt} Gen[g](t,\lambda(t)) \le C_0 F[E](t,\lambda(t))$. Integrating this last inequality in time and recalling $g(0,x,v) =f_0(x,v)$, we obtain the proposition. 	
\end{proof}

\begin{proposition}[Decay of electric field]\label{prop-decayE} 
	Let $\epsilon_0 = Gen[f_0](2\lambda_0)$ and $\theta_1>0$ as in \eqref{bootstrapE}. Under the boundedness of the generator norm \eqref{bd-g}, there holds
\begin{equation}\label{DecayE} 
F[E](t,\lambda(t)) \le C_2(\epsilon_0 + \epsilon^2) e^{-\theta_1 \langle t\rangle^{1-\delta}} 
\end{equation}
for some universal constant $C_2$.
\end{proposition}

We postpone to give the proof of Proposition \ref{prop-decayE} in the next section. We remark that a polynomial decay of the generator norm of $E(t)$ is sufficient for the nonlinear analysis as long as the shrinking of analyticity radius is in control, namely the inequality \eqref{cond-lambda} holds. Note however that an exponential decay of $E(t)$ is encoded in the generator norm, which is essential in suppressing plasma echoes.  
 
\begin{proof}[Proof of Theorem \ref{theo-LD}]
Let us now give the proof of Theorem \ref{theo-LD} using the results from Propositions \ref{prop-g}-\ref{prop-decayE}. Indeed, we may choose $\epsilon_0 \ll\epsilon\ll1$, if necessary, so that $C_2(\epsilon_0 + \epsilon^2) < \epsilon$, and therefore the equality in the bootstrap assumption \eqref{bootstrapE} never occurs, leading to global-in-time solutions that satisfy \eqref{bootstrapE} for all times $t\ge 0$. Next, we note that by definition of the $F$-norm \eqref{def-Nrho} and the fact that $A_{k,kt}(\lambda(t))\ge e^{\lambda_0 \langle k, kt\rangle} \ge e^{\lambda_0 t}$ for $k\in \ZZ^d \setminus\{0\}$, we have $\|E(t)\|_{L^\infty} \le e^{-\lambda_0 t} F[E](t,\lambda(t)) $, yielding the desired decay for the electric field. Finally, the scattering of $f$ is equivalent to the convergence of $g(t,x,v)$ in the large time, which follows from the boundedness of $g$ in the analytic norms \eqref{bd-g}.
This ends the proof of Theorem \ref{theo-LD}.
\end{proof}

%We first note the following, which in particular implies $F[\rho](t,z) \le G[g(t)]^{1/2}(z)$. 
%
%
%\begin{lemma} 
%For any $z\ge 0$, there holds 
%\begin{equation}\label{sup-g} \sup_{k,\eta}   | \Fg_{k,\eta}|  \le G[g](z)^{1/2}. \end{equation}
%\end{lemma}
%\begin{proof} Set $ A_{k,\eta}(z) = e^{z\langle k,\eta\rangle^\gamma} \langle k,\eta\rangle^{\sigma} $. We compute 
%$$
%\begin{aligned} 
% A^2_{k,\eta} | \Fg_{k,\eta}|^2 
%&\le A^2_{k,\eta}\int_{|\eta'| \ge |\eta|} |\Fg_{k,\eta'}||\partial_\eta \Fg_{k,\eta'}| \; d\eta' 
%\le \int_\RR A_{k,\eta'}^2 |\Fg_{k,\eta'}|
%| \partial_\eta \Fg_{k,\eta'}| \; d\eta' 
%\\& \le \sum_{k'\in \ZZ}  \int_\RR A_{k',\eta'}^2 \Big( |\Fg_{k',\eta'}|^2 + 
%| \partial_\eta \Fg_{k',\eta'}|^2 \Big)  \; d\eta',
%\end{aligned}
%$$ giving the lemma. 
%\end{proof}
%

%%%%%%%%%%%%%%%%%%%%%%%%%%

\subsection{Decay of electric field}\label{sec-decayE}

%%%%%%%%%%%%%%%%%%%%%%%%%%

It remains to prove Proposition \ref{prop-decayE} establishing decay of the electric field under the boundedness of the generator norm \eqref{bd-g}. Indeed, we first recall the electric formulation \eqref{nonE} that 
\begin{equation}\label{recallE}
E(t) = G\star_{t,x}E^0(t,x) + G\star_{t,x}E^R(t,x)
\end{equation}  	
in which $G(t,x)$ is the Green function for the linearized problem as established in Theorem \ref{theo-linLD}, while $E^0(t,x)$ is the free electric field generated by \eqref{free} with the same initial data $f_0(x,v)$, and $E^R=  -\nabla_x\Delta_x^{-1}\nabla_x \cdot \cE^R$, with $\cE^R$ being defined as in \eqref{ER}. We stress that we are treating the torus case with analytic equilibria, and therefore no oscillatory component occurs in the Green function: namely, there holds 
\begin{equation}
\label{re-Green}\FG_k(t) = \delta(t) + \FG^r_k(t), \qquad |\FG^r_k(t)| \le C \langle k\rangle^{-1}e^{-\theta_0\langle kt\rangle},
\end{equation}
uniformly in $k\in \ZZ^d\setminus\{0\}$, for some positive constants $\theta_0,C$. The constant $\theta_0$ depends on the analyticity radius of equilibria $\mu(v)$, which we assume that $\theta_0 >2\lambda_0$. 

\begin{lemma}\label{lem-GS} Let $G$ satisfy \eqref{re-Green}. Then, there holds
	\begin{equation}\label{bd-ES}\begin{aligned}
F[G\star_{t,x}S](t,\lambda(t)) \le F[S](t,\lambda(t)) + C\int_0^t e^{-2\theta_1 (t-s)}F[S](s,\lambda(s))\; ds 
\end{aligned}\end{equation}
for any source term $S(t,x)$ and for some positive constants $\theta_1,C$.  	
\end{lemma}
\begin{proof}
Indeed, recalling $F$-norms \eqref{def-Nrho} and using the decomposition \eqref{re-Green}, we compute 
	$$	
\begin{aligned}
F[G\star_{t,x}S](t,\lambda(t)) 
&\le  \sum_{k\in \ZZ^d\setminus \{0\}} A_{k,kt}(t) \Big[  |\FS_k(t)| + C \int_0^t e^{-\theta_0\langle |k|(t-s)\rangle} |\FS_k(s)| \; ds\Big],
\end{aligned}$$	
in which $A_{k,kt}(t)$ is the analytic weight with radius $z = \lambda(t)$. 
The first term is already $F[S](t,\lambda(t))$ by definition. As for the second term, we first use \eqref{Aweight} with $k'=k, \eta=kt$ and $\eta' = ks$ to bring the $A$-weight into the time integration, namely $A_{k,kt}(t) \le A_{k,ks}(t) A_{0,k(t-s)}(t)$. Next, since $\lambda(t)$ is decreasing in time, we bound $A_{k,ks}(t) \le A_{k,ks}(s)$, giving the right $A$-weight for $\FS_k(s)$, and since $\lambda(0) = 2\lambda_0$,  
$$A_{0,k(t-s)}(t) e^{-\theta_0\langle |k|(t-s)\rangle}\le e^{(2\lambda_0-\theta_0) \langle k(t-s)\rangle } \langle k(t-s)\rangle^\sigma \le C, $$ 
for some constant $C$, upon recalling that $\theta_0>2\lambda_0$. The proposition follows. 
\end{proof}

Effectively, Lemma \ref{lem-GS} reduces to bound $F[E^0]$ and $F[E^R]$, namely reducing to treat plasma echoes as if it were the exact same problem near vacuum. Let us start with the free electric field $E^0$. 

\begin{lemma}\label{lem-E0}
Let $E^0 = \nabla_x\Delta_x^{-1}\rho^0$ be the free electric field, and $\epsilon_0 = Gen[f_0](2\lambda_0)$. Then, there holds 	
\begin{equation}
\label{bdE0} F[E^0](t,\lambda(t)) \le \epsilon_0e^{-\theta_1 \langle t\rangle}. 
\end{equation}
\end{lemma}  
\begin{proof} By definition, $\rho^0 = \int f_0(x-vt,v)\; dv$ and so $\Frho^0_k(t) = \Ff_{0,k,kt}$. Hence, we compute 
		$$
	\begin{aligned}  F[E^0](t,\lambda(t))
		&\le \sum_{k\in \ZZ^d\setminus \{0\}}    |k|^{-1}A_{k,kt}(t) | \Ff_{0,k,kt}| .
	\end{aligned}$$
To bound the summation in term of the generator norm of $f_0$, noting $A_{k,\eta}$ is increasing in $|\eta|$, we have the following pointwise bound 
\begin{equation}\label{sup-g} 
\begin{aligned} 
A_{k,\eta}| \Fg_{k,\eta}| 
&\le A_{k,\eta} \sum_{|\alpha|\le d}\int_{|\eta'| \ge |\eta|} |\partial^\alpha_\eta \Fg_{k,\eta'}| \; d\eta' 
\le \sum_{|\alpha|\le d}\int_{\RR^d} A_{k,\eta'} | \partial^\alpha_\eta \Fg_{k,\eta'}| \; d\eta' ,
\end{aligned}\end{equation}
for any $k,\eta$. The lemma thus follows, upon taking $\eta=kt$ and $g = f_0$ in the above estimate, and noting $\lambda(t) \le \lambda(1) < 2\lambda_0$ for $t\ge 1$ and so $A_{k,\eta}(t) \le A_{k,\eta}(2\lambda_0) e^{-\theta_1\langle t\rangle} $ for $k\not =0$. 
\end{proof}
 
\begin{lemma}
	\label{lem-ER}
Let	$E^R=  -\nabla_x\Delta_x^{-1}\nabla_x \cdot \cE^R$, with $\cE^R$ being defined as in \eqref{ER}. Then, as long as the bootstrap assumption \eqref{bootstrapE} remains valid, there holds 
\begin{equation}
\label{bdER} F[E^R](t,\lambda(t)) \le C_0\epsilon (\epsilon_0 + C_1 \epsilon) e^{-\theta_1 \langle t\rangle^{1-\delta}} 
\end{equation}
for some constants $C_0, C_1$. 
\end{lemma}

\begin{proof} By definition, we note that $|\FE_k^R|\le |\FcE_k^R|$ and so $ F[E^R]\le F[\cE^R]$, where $\FcE^R_k(t)$ is computed as in \eqref{ER-Fr}. Recall that $f(t,x+vt,v) = g(t,x,v)$, and so $\Ff_{k,\eta}(t) = \Fg_{k,\eta+kt}$. Putting this into \eqref{ER-Fr}, we thus have 
\begin{equation}
\label{ER-Fr1}
\FcE^R_k(t) = \sum_{\ell \in \ZZ^d \setminus\{0\}}\int_0^t (t-s) \FE_{\ell}(s) \Fg_{k-\ell, kt - \ell s}(s)ds.
\end{equation}
Let us now compute 
	$$ 
	\begin{aligned}
	F[\cE^R](t,\lambda(t)) 
	& \le  
 \sum_{k,\ell\in \ZZ^d\setminus \{0\}} A_{k,kt}(t) \int_0^t  (t-s)   \FE_\ell(s) \Fg_{k-\ell,kt-\ell s}(s)\; ds ,
	\end{aligned}$$
in which we recall that $A_{k,\eta}(t) = e^{\lambda(t)\langle k,\eta\rangle} \langle \eta\rangle^{\sigma} $. For convenience, set 	\begin{equation}\label{def-CC}C_{k,\ell}(t,s) : =   \frac{(t-s)A_{k,kt}(t)}{A_{\ell,\ell s}(s) A_{k-\ell, kt-\ell s}(s)}. \end{equation}
Using the pointwise estimate \eqref{sup-g}, we bound 
	$$ \begin{aligned}
	F[\cE^R](t,\lambda(t)) 
	& \le
	\sum_{k,\ell\not =0} \int_0^t  C_{k,\ell}(t,s)
	A_{\ell,\ell s}(s) |\FE_\ell(s)| A_{k-\ell, kt-\ell s}(s) |\Fg_{k-\ell,kt-\ell s}(s) | \; ds 
	\\
	& \le
\sum_{|\alpha|\le d}	\sum_{k,\ell\not =0} \int_0^t C_{k,\ell}(t,s) A_{\ell,\ell s}(s) |\FE_\ell(s)|
	\int_{\RR^d} A_{k-\ell,\eta}(s)|\partial^\alpha_\eta \Fg_{k-\ell,\eta}(s) |  ds
	\\
	& \le
	\int_0^t \sup_{k,\ell\not=0}C_{k,\ell}(t,s) F[E](s,\lambda(s)) 
	Gen[g](s,\lambda(s)) ds.
	\end{aligned}$$
Using the bootstrap assumption \eqref{bootstrapE} and the estimate \eqref{bd-g}, we therefore obtain 
	$$ \begin{aligned}
F[\cE^R](t,\lambda(t)) 
& \le \epsilon (\epsilon_0 + C_1 \epsilon)
\int_0^t \sup_{k,\ell\not=0}C_{k,\ell}(t,s)  e^{-\theta_1 \langle s\rangle^{1-\delta}}  ds.
\end{aligned}$$
Lemma \ref{lem-ER} thus follows from the following claim 
	\begin{equation}\label{claimCR}
	e^{\theta_1\langle t\rangle^{1-\delta}} \int_0^t \sup_{k,\ell\not=0}C_{k,\ell}(t,s) e^{-\theta_1\langle s\rangle^{1-\delta}} ds \le C_0 .
	\end{equation}
Let us now prove the above claim, namely the suppression of echoes, noting the factor $(t-s)$ present in $C_{k,\ell}(t,s)$. Indeed, by definition of $A_{k,kt}$, we note that 
	$$ 
	\begin{aligned}
	C_{k,\ell}(t,s) 
	&\le (t-s)  e^{(\lambda(t)-\lambda(s)) \langle k,kt\rangle} 
	\langle kt\rangle^\sigma \langle kt - \ell s \rangle^{-\sigma} \langle \ell s \rangle^{-\sigma} 
	\\ &\le (t-s)  e^{(\lambda(t)-\lambda(s)) \langle k,kt\rangle}  \min \Big\{ \langle kt - \ell s \rangle, \langle \ell s \rangle\Big \} ^{-\sigma}.
	\end{aligned}$$
	Recall that $\lambda(t) = \lambda_0 (1 + \langle t\rangle^{-\delta})$, which we may bound 
	$|\lambda(s) - \lambda(t)|\ge 2\theta_2|t-s|/\langle t\rangle^{1+ \delta}$ for some positive constant $\theta_2$. Therefore, we have 
	\begin{equation}\label{bd-C1}
	\begin{aligned}
	C_{k,\ell}(t,s) 
	&\le C_0 (t-s) e^{-2\theta_2 |k(t-s)|/ \langle t\rangle^{\delta}}
	\min \Big\{ \langle kt - \ell s \rangle, \langle \ell s \rangle\Big \} ^{-\sigma} .
	\end{aligned}\end{equation}
We stress that the gain of an exponential factor is precisely due to the slowly shrinking of analyticity radius. To prove the claim \eqref{claimCR}, we first take care of the exponential term $e^{\theta_1\langle t\rangle^{1-\delta}}$. Indeed, using one half of the exponential gain in \eqref{bd-C1}, since $k\not =0$, and noting $ \langle t\rangle^{1-\delta} - \langle s\rangle^{1-\delta} \le |t-s| / \langle t\rangle^\delta $ for $0\le s\le t$, we bound 
\begin{equation}
\label{exp-bd}
e^{\theta_1\langle t\rangle^{1-\delta}} e^{-\theta_2 |t-s| / \langle t\rangle^{\delta}} e^{-\theta_1\langle s\rangle^{1-\delta}}  \le 1,
\end{equation}  
provided $\theta_1\le \theta_2$. Recall that $\theta_1$ appears in the bootstrap assumption \eqref{bootstrapE}, and therefore, can be taken smaller, if needed, so that $\theta_1\le \theta_2$. Therefore, the claim \eqref{claimCR} would follow, provided that 
\begin{equation}\label{claimCR1}
\int_0^t (t-s) e^{-\theta_2 |k(t-s)|/ \langle t\rangle^{\delta}}
\min \Big\{ \langle kt - \ell s \rangle, \langle \ell s \rangle\Big \} ^{-\sigma} ds \le C_0 ,
\end{equation}
for $k, \ell\not =0$. Namely, we need to bound the growth factor $(t-s)$. To this end, we consider two cases: $s\le t/2$ and $s\ge t/2$. In the former case, we use the exponential gain, namely $e^{-\theta_2 |k(t-s)|/ \langle t\rangle^{\delta}} \le e^{-\frac12\theta_2 \langle t\rangle^{1-\delta}}$, which proves \eqref{claimCR1} in this case. On the other hand, when $s\ge t/2$, we observe that if either 
$$\min\Big\{ \langle kt - \ell s \rangle, \langle \ell s \rangle \Big\} \ge \langle \ell s \rangle \ge \frac t2$$
or $$\min\Big\{ \langle kt - \ell s \rangle, \langle \ell s \rangle \Big\} \ge \langle kt - \ell s \rangle \ge \frac t4$$
holds, the claim \eqref{claimCR1} again follows, since $\ell\not =0$, $s\ge t/2$, and $\sigma\ge 1$. Finally, it remains to consider the case when $|kt - \ell s|\le t/4$. Note that the case when $k = \ell$ clearly gives \eqref{claimCR1}, using the Sobolev weight $\langle k(t-s)\rangle^{-\sigma}$. For $k\not =\ell$ (and $s\ge t/2$), we bound 
\begin{equation}\label{aecho}|k(t-s)| \ge |k-\ell|s - |kt - \ell s| \ge s - \frac t4 \ge \frac t4 ,\end{equation}
which again gives an exponential decay, since 
$e^{-\theta_2 |k(t-s)|/ \langle t\rangle^{\delta}}\le e^{-\frac14\theta_2 \langle t\rangle^{1-\delta}} $ in this case. This ends the proof of the lemma. 
\end{proof}

\begin{proof}[Proof of Proposition \ref{prop-decayE}]
The proposition, establishing decay of the electric field, now follows straightforwardly from the above lemmas. Indeed, from the formulation \eqref{recallE} and Lemma \ref{lem-GS}, we bound 
$$
\begin{aligned}
F[E](t, \lambda(t)) &\le F[G\star_{t,x}E^0](t,\lambda(t)) + F[G\star_{t,x}E^R](t,\lambda(t))
\\
& \le F[E^0](t,\lambda(t)) + F[E^R](t,\lambda(t))+  C\int_0^t e^{-2\theta_1 (t-s)}\Big(F[E^0] + F[E^R]\Big)(s,\lambda(s))\; ds.
\end{aligned} 
$$
Next, by using Lemmas \ref{lem-E0}-\ref{lem-ER}, the first two terms are already treated as desired, while the integral term is bounded by 
$$ 
C_0(\epsilon_0 + C_1 \epsilon^2)  \int_0^t e^{-2\theta_1 (t-s)} e^{-\theta_1 \langle s\rangle^{1-\delta}} \; ds 
\le C_0(\epsilon_0 + C_1 \epsilon^2) e^{-\theta_1 \langle t\rangle^{1-\delta}},$$
upon noting $e^{-\theta_1 (t-s)}\le e^{-\theta_1 (t-s)/\langle t\rangle^\delta}$ and using again the inequality \eqref{exp-bd} with $\theta_2=\theta_1$. This ends the proof of Proposition \ref{prop-decayE}, and therefore the main theorem. 
\end{proof}

\section{Landau damping for Vlasov-Riesz systems}

In this section, we establish the nonlinear Landau damping near stable Penrose equilibria for a family of Vlasov systems with Riesz-type interaction potentials, namely the Vlasov equation \eqref{VP1}, or the perturbed system \eqref{VP-pert} near $\mu(v)$, where the electric field is given by \eqref{VP2Riesz} for any $\alpha \in [0,2]$. Our result reads
 
 \begin{theorem}[Nonlinear Landau damping for Vlasov-Riesz systems]\label{theo-LDR}
	Let $\lambda_0 > 0$, $\alpha \in [0,2]$, and let $\mu(v)$ be any spatially homogenous, real analytic, Penrose stable equilibria. Then, for any initial data of the form $\mu(v) + f_0$ with  
	$$ 
\epsilon_0:=	\sum_{|\alpha|\le d}\sum_{k\in \ZZ^d}  \int_{\RR^d}  e^{2\lambda_0 \langle k,\eta\rangle} 
	|\partial^\alpha_\eta \widehat f_{0,k,\eta}| d \eta 
	$$
	being sufficiently small, 
	the nonlinear Landau damping for the Vlasov-Riesz system \eqref{VP1}-\eqref{VP2Riesz} holds: $\|E(t)\|_{L^\infty}\lesssim \epsilon_0 e^{-\lambda_0 t}$ for all $t\ge 0$, and $f(t,x+vt,v)$ converges to a limit $f_\infty(x,v)$ in the large time.
\end{theorem}

We stress that the results established in Theorem \ref{theo-LDR} are new for the Vlasov-Riesz systems, and surprisingly, include the borderline Vlasov-Dirac-Benney system when $\alpha =0$ in the sharp analytic spaces. The linear analysis has been studied in \cite{BardosBesse, DanielToan}, see also \cite{HanKwanR}, however the nonlinear Landau damping appears open, which we resolve in Theorem \ref{theo-LDR}.  

\begin{proof}[Proof of Theorem \ref{theo-LDR}]
The proof follows identically to that of Theorem \ref{theo-LD}, upon noting that the nonlinear electric field $E^R$ in \eqref{Eform}-\eqref{ER} now reads
\begin{equation}
\label{ERR}
E^R = \nabla_x (-\Delta_x)^{-\alpha/2}\nabla_x\cdot \int_0^t \int_{\RR^3} (t-s)[Ef](s,x-v(t-s),v)\; dvds,
\end{equation}
and therefore, in Fourier variables, reads
\begin{equation}
\label{ER-Fr1R}
\FE^R_k(t) = - |k|^{-\alpha}k^{\otimes2} \sum_{\ell \in \ZZ^d \setminus\{0\}}\int_0^t (t-s) \FE_{\ell}(s) \Fg_{k-\ell, kt - \ell s}(s)ds,
\end{equation}
which may a priori face a loss of two derivatives (e.g. for the case of Vlasov-Dirac-Benney systems with $\alpha =0$). In view of the proof of Lemma \ref{lem-ER}, it suffices to show that the time integral, see \eqref{claimCR} and \eqref{claimCR1}, gains two derivatives: precisely, we claim that 
\begin{equation}\label{claimCR1again}
\int_0^t (t-s) e^{-\frac12\theta_2 |k(t-s)|/ \langle t\rangle^{\delta}}
\min \Big\{\langle \ell, \ell s \rangle,  \langle k-\ell, kt - \ell s \rangle\Big \} ^{-\sigma} ds \le C_0 |k|^{-2},
\end{equation}
for $k, \ell\not =0$. Indeed, using the fact that $xe^{-x}\lesssim 1$ for $x\ge 0$, we may bound    
\begin{equation}\label{intk2}
\int_0^t (t-s) e^{-\frac14\theta_2 |k(t-s)|/ \langle t\rangle^{\delta}} ds \lesssim \langle t\rangle^{2\delta} |k|^{-2},
\end{equation}
upon introducing the change of variable $\theta = k(t-s)$. We stress that the above time integration gains a factor of $|k|^{-2}$ on the right hand side. That is, the claim \eqref{claimCR1again} would follow, provided that 
\begin{equation}\label{claimCR2}
\sup_{0\le s\le t} e^{-\frac14\theta_2 |k(t-s)|/ \langle t\rangle^{\delta}}
\min \Big\{\langle \ell, \ell s \rangle,  \langle k-\ell, kt - \ell s \rangle\Big \} ^{-\sigma} \le C_0\langle t\rangle^{-2\delta} 
\end{equation}
uniformly for $k,\ell\not =0$ and $t\ge 0$. The claim \eqref{claimCR2} clearly holds, if $s\le t/2$, since the exponential gain yields a rapid decay $e^{-\frac12\theta_2 |k(t-s)|/ \langle t\rangle^{\delta}} \le e^{-\frac14\theta_2 \langle t\rangle^{1-\delta}}$, noting $k\not =0$. On the other hand, in the case when $s\ge t/2$, then we have two cases: $\min\Big\{\langle \ell, \ell s \rangle ,  \langle k-\ell, kt - \ell s \rangle\Big\} \ge \langle \ell, \ell s \rangle$ or $\min\Big\{\langle \ell, \ell s \rangle ,  \langle k-\ell, kt - \ell s \rangle\Big\} \ge \langle k-\ell, kt-\ell s \rangle$ with $|kt-\ell s|\ge t/4$, each of which yields decay of order $\langle t\rangle^{-\sigma}$, proving \eqref{claimCR2}, since $\delta \ll1$. The case when $|kt - \ell s|\le t/4$ and $k\not =\ell$, is treated as before, see \eqref{aecho}, using the exponential gain 
$e^{-\frac14\theta_2 |k(t-s)|/ \langle t\rangle^{\delta}}\le e^{-\frac1{16}\theta_2 \langle t\rangle^{1-\delta}} $, yielding  \eqref{claimCR2}. Finally, in the case when $k=\ell$, we gain $\langle k(t-s)\rangle^{-\sigma} $ from the Sobolev weight, and therefore, instead of \eqref{intk2}, we bound
$$\int_0^t (t-s) \langle k(t -s) \rangle^{-\sigma} ds \le C_0|k|^{-2},
$$
proving again \eqref{claimCR1again} in this case. 
This completes the proof of 
Theorem \ref{theo-LDR}.
\end{proof}

\subsection*{Acknowledgement}
{\em The author would like to thank Daniel Han-Kwan and Frederic Rousset for their kind suggestions to include the borderline Vlasov-Dirac-Benney system. The research is supported in part by the NSF under grant DMS-2349981. The author would like to acknowledge the hospitality of the Institut des Hautes \'Etudes Scientifiques for a visit during which part of this work was
carried out.}

%%%%%%%%%%%%%%%%%
%%%%%%%%%%%%%%%%%
\bibliographystyle{abbrv}
%\bibliography{Landau}

\begin{thebibliography}{10}
	
	\bibitem{BardosBesse} C.~W. Bardos and N. Besse, The Cauchy problem for the Vlasov-Dirac-Benney equation and related issues in fluid mechanics and semi-classical limits, {\em Kinet. Relat. Models} {\bf 6} (2013), no.~4, 893--917.
	
	\bibitem{BD}
	C.~Bardos and P.~Degond.
	\newblock Global existence for the {Vlasov}-{Poisson} equation in 3 space
	variables with small initial data.
	\newblock {\em Ann. Inst. Henri Poincar{\'e}, Anal. Non Lin{\'e}aire},
	2:101--118, 1985.
	
	\bibitem{Bed1}
	J.~Bedrossian.
	\newblock Suppression of plasma echoes and {{L}andau} damping in {Sobolev}
	spaces by weak collisions in a {Vlasov}-{Fokker}-{Planck} equation.
	\newblock {\em Ann. PDE}, 3(2):66, 2017.
	\newblock Id/No 19.
	
	\bibitem{Bed2}
	J.~Bedrossian.
	\newblock Nonlinear echoes and {{L}andau} damping with insufficient regularity.
	\newblock {\em Tunis. J. Math.}, 3(1):121--205, 2021.
	
	\bibitem{BMM-apde}
	J.~Bedrossian, N.~Masmoudi, and C.~Mouhot.
	\newblock {L}andau damping: paraproducts and {Gevrey} regularity.
	\newblock {\em Ann. PDE}, 2(1):71, 2016.
	\newblock Id/No 4.
	
	\bibitem{BMM-cpam}
	J.~Bedrossian, N.~Masmoudi, and C.~Mouhot.
	\newblock {L}andau damping in finite regularity for unconfined systems with
	screened interactions.
	\newblock {\em Commun. Pure Appl. Math.}, 71(3):537--576, 2018.
	
	\bibitem{BMM-lin}
	J.~Bedrossian, N.~Masmoudi, and C.~Mouhot.
	\newblock Linearized wave-damping structure of {Vlasov}-{Poisson} in
	{{\(\mathbb{R}^3\)}}.
	\newblock {\em SIAM J. Math. Anal.}, 54(4):4379--4406, 2022.
	
	\bibitem{CLN}
	S.~Chaturvedi, J.~Luk, and T.~T. Nguyen.
	\newblock {The Vlasov--Poisson--Landau system in the weakly collisional
		regime}.
	\newblock {\em Journal of the American Mathematical Society}, 2023.
	
	\bibitem{GIa} A. Gagnebin and M. Iacobelli. 
	\newblock Landau damping on the torus for the {V}lasov-{P}oisson system with massless electrons.
	\newblock{\em J. Differential Equations}, (376):154--203, 2023.	

	\bibitem{Gould}
	R.~W. Gould, T.~M. O'Neil, and J.~H. Malmberg.
	\newblock Plasma wave echo.
	\newblock {\em Phys. Rev. Lett.}, 19(5), 1967.
	
	\bibitem{GNR1}
	E.~Grenier, T.~T. Nguyen, and I.~Rodnianski.
	\newblock {L}andau damping for analytic and {Gevrey} data.
	\newblock {\em Math. Res. Lett.}, 28(6):1679--1702, 2021.
	
	\bibitem{GNR2}
	E.~Grenier, T.~T. Nguyen, and I.~Rodnianski.
	\newblock Plasma echoes near stable {P}enrose data.
	\newblock {\em SIAM J. Math. Anal.}, 54(1):940--953, 2022.
	
	\bibitem{DanielToan} D. Han-Kwan and T.~T. Nguyen, Ill-posedness of the hydrostatic Euler and singular Vlasov equations, 
	{\em Arch. Ration. Mech. Anal.} {\bf 221} (2016), no.~3, 1317--1344. 
		
	\bibitem{HKNR2}
	D.~Han-Kwan, T.~T. Nguyen, and F.~Rousset.
	\newblock Asymptotic stability of equilibria for screened {Vlasov}-{Poisson}
	systems via pointwise dispersive estimates.
	\newblock {\em Ann. PDE}, 7(2):37, 2021.
	\newblock Id/No 18.
	
	\bibitem{HKNR3}
	D.~Han-Kwan, T.~T. Nguyen, and F.~Rousset.
	\newblock On the linearized {Vlasov}-{Poisson} system on the whole space around
	stable homogeneous equilibria.
	\newblock {\em Commun. Math. Phys.}, 387(3):1405--1440, 2021.
	
\bibitem{HKNR4}
D.~Han-Kwan, T.~T. Nguyen, and F.~Rousset.
\newblock Linear {L}andau damping for the {Vlasov}-{Maxwell} system in
  $\mathbb{R}^3$.
\newblock {\em Ann. PDE} {\bf 11} (2025), no.~2, Paper No. 26.

	\bibitem{HanKwanR} D. Han-Kwan and F. Rousset, Quasineutral limit for Vlasov-Poisson with Penrose stable data, {\em Ann. Sci. \'Ec. Norm. Sup\'er.} (4) {\bf 49} (2016), no.~6, 1445--1495
		

	\bibitem{IPWW2}
	A.~D. Ionescu, B.~Pausader, X.~Wang, and K.~Widmayer.
	\newblock On the stability of homogeneous equilibria in the {Vlasov}-{Poisson}
	system on {{\(\mathbb{R}^3\)}}.
	\newblock {\em Classical Quantum Gravity}, 40(18):32, 2023.
	\newblock Id/No 185007.
	
	
	
	\bibitem{IPWW3}
	A.~D. Ionescu, B.~Pausader, X.~Wang, and K.~Widmayer.
	\newblock Nonlinear {L}andau damping and wave operators in sharp gevrey spaces.
	\newblock {\em arXiv preprint arXiv:2405.04473}, 2024.
	
	\bibitem{IPWW}
	A.~D. Ionescu, B.~Pausader, X.~Wang, and K.~Widmayer.
	\newblock Nonlinear {L}andau damping for the {V}lasov-{P}oisson system in
	{$\Bbb R^3$}: the {P}oisson equilibrium.
	\newblock {\em Ann. PDE}, 10(1):Paper No. 2, 78, 2024.
	
	\bibitem{Landau}
	L.~Landau.
	\newblock On the vibrations of the electronic plasma.
	\newblock {\em Akad. Nauk SSSR. Zhurnal Eksper. Teoret. Fiz.}, 16:574--586,
	1946.
	
	\bibitem{LionsPerthame}
	P.-L. Lions and B.~Perthame.
	\newblock Propagation of moments and regularity for the {$3$}-dimensional
	{V}lasov-{P}oisson system.
	\newblock {\em Invent. Math.}, 105(2):415--430, 1991.
	
	\bibitem{MV}
	C.~Mouhot and C.~Villani.
	\newblock On {{L}andau} damping.
	\newblock {\em Acta Math.}, 207(1):29--201, 2011.
	
	\bibitem{TrinhLD}
	Trinh.~T. Nguyen.
	\newblock Derivative estimates for screened {V}lasov-{P}oisson system around
	{P}enrose-stable equilibria.
	\newblock {\em Kinet. Relat. Models}, 13(6):1193--1218, 2020.
	
	\bibitem{Toan}
	T.~T. Nguyen.
	\newblock {{L}andau damping and survival threshold}.
	\newblock {\em arXiv:2305.08672}, 2023.
	
	\bibitem{ToanVKG2}
	T.~T. Nguyen.
	\newblock Landau damping below survival threshold.
	\newblock {\em arXiv:2412.18620}, 2024.
	
	\bibitem{ToanVKG1}
	T.~T. Nguyen.
	\newblock A new framework for particle-wave interaction.
	\newblock {\em arXiv:2410.13703}, 2024.
	
	\bibitem{Pfaffelmoser}
	K.~Pfaffelmoser.
	\newblock Global classical solutions of the {V}lasov-{P}oisson system in three
	dimensions for general initial data.
	\newblock {\em J. Differential Equations}, 95(2):281--303, 1992.
	
	\bibitem{Schaeffer}
	J.~Schaeffer.
	\newblock Global existence of smooth solutions to the {V}lasov-{P}oisson system
	in three dimensions.
	\newblock {\em Comm. Partial Differential Equations}, 16(8-9):1313--1335, 1991.
	
	\bibitem{Trivelpiece}
	A.~Trivelpiece and N.~Krall.
	\newblock {\em Principles of Plasma Physics}.
	\newblock McGraw-Hill, Englewood Cliffs, 1973.
	
\end{thebibliography}

\end{document}